\documentclass[12pt]{amsart}
\usepackage{amssymb,amsmath,amsthm}
\usepackage{tikz}
\usepackage{cancel}
\usepackage{mathtools,}

\usetikzlibrary{arrows,shapes,automata,backgrounds,decorations,petri,positioning}

\theoremstyle{plain}
\newtheorem{thm}{Theorem}[section]
\newtheorem{lem}[thm]{Lemma}

\newtheorem{cor}[thm]{Corollary}

\newtheorem{prop}[thm]{Proposition}
\theoremstyle{definition}
\newtheorem{defn}[thm]{Definition}

\newtheorem{ex}[thm]{Example}

\newcommand{\ts}[1]{\mbox{\textnormal{\textsf{#1}}}}

\newcommand{\ol}[1]{\overline{#1}}
\newcommand{\ub}[1]{\underbracket[.5pt][1pt] {#1}}
\newcommand{\bond}[1]{\textup{\textsf{V}}({#1})}
\renewcommand{\bar}[1]{\textup{\textsf{B}}({#1})}
\newcommand{\class}[1]{\|{#1}\|}
\newcommand{\av}{\textup{Av}}
\newcommand{\pt}{\ast}
\renewcommand{\b}{\textup{\textsf{b}}}
\renewcommand{\v}{\textup{\textsf{v}}}

\title{Coincidental pattern avoidance}

\author{Bridget Eileen Tenner}
\address{Department of Mathematical Sciences, DePaul University, Chicago, IL 60614}
\email{bridget@math.depaul.edu}

\thanks{Research partially supported by a DePaul University Competitive Research Leave Grant.} 


\begin{document}

\begin{abstract}
There are several versions of permutation pattern avoidance that have arisen in the literature, and some known examples of two different types of pattern avoidance coinciding.  In this paper, we examine barred patterns and vincular patterns.  Answering a question of Steingr\'imsson, we determine when barred pattern avoidance coincides with avoiding a finite set of vincular patterns, and when vincular pattern avoidance coincides with avoiding a finite set of barred patterns.  There are $720$ barred patterns with this property, each having between $3$ and $7$ letters, of which at most $2$ are barred, and there are $48$ vincular patterns with this property, each having between $2$ and $4$ letters and exactly one bond.\\

\end{abstract}

\maketitle

\section{Introduction}

A number of phenomena are equivalent to pattern avoidance; that is, the objects possessing a particular property $P$ can be described by permutations, and these permutations are precisely the permutations that avoid a particular set of patterns $S(P)$.  For example, the permutations whose principle order ideals in the Bruhat order are boolean are exactly those permutations that avoid the two patterns $321$ and $3412$ \cite{tenner patt-bru}.  For classical pattern avoidance, such phenomena are catalogued in \cite{dppa}. 

The classical idea of pattern avoidance began with the work of Simion and Schmidt \cite{simion schmidt}.  Since then, variations on the theme have been developed, including \emph{barred} patterns, \emph{vincular} patterns, \emph{bi-vincular} patterns, \emph{mesh} patterns, \emph{marked mesh} patterns, and \emph{Bruhat-restricted} patterns.  The purpose of the current paper is to describe the relationship between the first two of these variations: barred patterns and vincular patterns.  Barred patterns were first introduced by West \cite{west-thesis}, where he showed that two-stack sortable permutations are exactly those that avoid a particular classical pattern on $4$ letters and a barred pattern on $5$ letters.  Vincular patterns, also called ``generalized'' or ``dashed'' patterns, were first studied systematically by Babson and Steingr\'imsson in \cite{babson steingrimsson}, where many Mahonian permutation statistics were shown to be linear combinations of vincular patterns.  Vincular patterns were also surveyed in \cite{steingrimsson}.  (For the sake of completeness, we point out that the other types of patterns mentioned above can be found in \cite{bousquet-melou claesson dukes kitaev, branden claesson, ulfarsson, woo yong}, respectively.  More generally, patterns in both permutations and words have been explored in the text \cite{kitaev}.)

Perhaps surprisingly, given their definitions, there are sporadic instances of different types of pattern avoidance coinciding.  For example,
\begin{eqnarray*}
\av(\ub{13}2) &=& \av(132),\\
\av(41\ol{3}52) &=& \av(3\ub{14}2), \text{ and}\\
\av(63\ol{1}\ol{7}524) &=& \av(5\ub{24}13,\ 63\ub{15}24,\ 5\ub{26}413)
\end{eqnarray*}
(see Lemma~\ref{lem:bondable ex} and Example~\ref{ex:coincidental}).  On the other hand, for any classical permutation $p$, there is no set of barred permutations $\ts{B}$ for which $\av(p)$ and $\av(\ts{B})$ are the same set (see Lemma~\ref{lem:barred not equiv to classical}).

In this paper, we answer a question of Steingr\'imsson, posed in \cite{steingrimsson}, to characterize when barred pattern avoidance can be mimicked by avoiding a set of vincular patterns.  We also answer the symmetric question, of when vincular pattern avoidance can be mimicked by avoiding a set of barred patterns.  These answers appear in Theorem~\ref{thm:main bv} and Corollary~\ref{cor:main vb}, respectively, and it turns out that there are exactly $720$ barred patterns that can be mimicked in this way, and $48$ vincular patterns that can be.  The elements in this latter category are listed explicitly in Table~\ref{table:co vinc patterns}.

Section~\ref{section:flavors} is devoted to carefully defining barred and vincular patterns, giving examples of each, and exploring basic properties.  In Sections~\ref{section:prep vb} and~\ref{section:prep bv}, we lay the groundwork for stating and proving the main results of the paper, which occur in Sections~\ref{section:main} and~\ref{section:proof}, respectively.  Finally, in the last section, we suggest two directions for further research on this topic.

\section{Flavors of pattern avoidance}\label{section:flavors}

When studying permutation patterns, it is most useful to write permutations in one-line notation, although it should be noted that a relationship between patterns and reduced decompositions was shown in \cite{tenner rdpp}.  Throughout this paper, all permutations will be written in one-line notation.

\begin{defn}
For a positive integer $n$, the set of integers $\{1, \ldots, n\}$ is denoted $[n]$.
\end{defn}

\begin{ex}
The permutation $52134$ is the automorphism of $[5]$ defined by $1 \mapsto 5$, $2 \mapsto 2$, $3 \mapsto 1$, $4\mapsto 3$, and $5 \mapsto 4$.
\end{ex}

The classical concept of pattern avoidance is that a permutation $w$ \emph{contains} a \emph{pattern} $p$ if there is a subsequence of the one-line notation of $w$ that is in the same relative order as the letters in $p$.  If $w$ does not contain $p$, then $w$ \emph{avoids} $p$.

\begin{ex}\
\begin{itemize}
\item The permutation $52134$ contains the pattern $312$ because the letters in the subsequence $513$ are in the same relative order as the letters in the pattern $312$.  There are also three other occurrences of the pattern $312$ in $52134$: $514$, $523$, and $524$.
\item The permutation $52134$ avoids the pattern $1234$ because it does not have an increasing subsequence of length $4$.
\end{itemize}
\end{ex}

Patterns, whether they be classical, barred, or vincular, are concerned with the relative order of values in a permutation.  Thus, throughout this paper, order isomorphic permutations of subsets of $\mathbb{R}$ will be considered equivalent.  This equivalence will be denoted ``$\approx$,'' or ``='' when no confusion can arise.

\begin{ex}
As permutations of subsets of $\mathbb{R}$, we have $3\ 1\ 2 \ \approx \ \sqrt{5} \ -\!1 \ 0$.
\end{ex}

The two pattern generalizations that we study in this work are barred patterns and vincular patterns.  Although this paper examines multiple types of patterns simultaneously, it will be obvious from the notation whether a given pattern is classical (no decoration), barred (bars over some symbols), or vincular (brackets beneath some symbols).  Given a barred or vincular pattern $p$, the classical pattern underlying $p$, obtained by ignoring all decorations, will be denoted $\class{p}$.

\begin{defn}\label{defn:barred}
A \emph{barred} pattern is a permutation in which a subset of the letters have bars written over them.  A barred pattern is \emph{proper} if bars cover a proper subset of its letters.  A permutation $w$ contains a barred pattern $\b$ if $w$ contains an occurrence of the unbarred portion of $\b$ that is not also part of an occurrence of an entire $\class{\b}$-pattern.  Otherwise --- that is, if each occurrence of the unbarred portion of $\b$ in $w$ is part of a $\class{\b}$-pattern --- the permutation $w$ avoids $\b$.
\end{defn}

Henceforth, all barred patterns will be assumed to be proper, and simply called ``barred.''

\begin{ex}
Consider the barred pattern $\b = 3\ol{2}4\ol{1}$, where the unbarred portion is $34 \approx 12$.  For a permutation $w$ to contain $\b$, some $12$-pattern in $w$ must not be part of a $3241$-pattern in $w$.  The permutation $42351$ contains $\b$ because the increasing subsequence $23$ is not part of any $3241$-pattern.  On the other hand, the permutation $43251$ avoids $\b$.
\end{ex}

Generalizing the idea of classical patterns in a different direction, vincular patterns allow some letters in the pattern to be forced to be consecutive, or ``bonded'' together (hence the terminology). 

\begin{defn}\label{defn:vincular}
A \emph{vincular} pattern is a classical permutation in which consecutive symbols. including the left and right endpoints (each of which is denoted ``$*$''), may be bonded together.  A vincular pattern is \emph{proper} if there is at least one bond, and if the non-$*$ symbols are not all bonded together.  A permutation $w$ contains a vincular pattern $\v$ if $w$ contains a $\class{\v}$-pattern in which any substring bonded together in $\v$ must appear consecutively in $w$.  If the left (respectively, right) ``$\pt$'' is bonded to its adjacent symbol in $\v$, then the corresponding letter in $w$ must appear in the leftmost (respectively, rightmost) position of $w$.  Otherwise, the permutation $w$ avoids $\v$.
\end{defn}

Henceforth, all vincular patterns will be assumed to be proper, and simply called ``vincular.''

In a vincular pattern, brackets are drawn to indicate the bonds.

\begin{ex}\
\begin{itemize}
\item Consider the vincular pattern $\v = 3\ub{14}2$.  The permutation $41532$ contains $\v$ in two ways: $4153$ and $4152$.  The permutation $41352$ avoids $\v$ because, although $4152$ is a $\class{\v}$-pattern in $41352$, the $1$ and the $5$ are not consecutive.
\item Consider the vincular pattern $\v = \ub{\pt 3} \ub{14} 2$.  The permutation $41532$ contains $\v$ in two ways: $4153$ and $4152$.  The permutation $24153$ avoids $\v$ because the single occurrence of $\class{\v}$ in $w$ does not begin in the leftmost position of $24153$.
\end{itemize}
\end{ex}

In general, we define the set of permutations avoiding a pattern as follows.

\begin{defn}
For a (classical, barred, or vincular) pattern $p$, let $\av(p)$ be the set of permutations that avoid $p$.  Similarly, if $P$ is a collection of patterns (possibly of different types), then $\av(P)$ is the set of permutations simultaneously avoiding all patterns in $P$:
$$\av(P) = \bigcap_{p \in P} \av(p).$$
It will be helpful to note that for any $p \in P$,
$$\av(P) \subseteq \av(p).$$
\end{defn}

The following lemma follows immediately from Definitions~\ref{defn:barred} and~\ref{defn:vincular}.

\begin{lem}\label{lem:x in av(x)}
Let $\b$ be a barred pattern and $\v$ a vincular pattern.  Then $\class{\b} \in \av(\b)$ and $\class{\v} \not\in \av(\v)$.
\end{lem}

Barred and vincular patterns have arisen in numerous contexts in the literature.  We highlight a few of these here.  Precise definitions of the object they characterize are not relevant to this work, and the interested reader is referred to the citations given.

\begin{ex}\
\begin{itemize}
\item The two-stack sortable permutations are $\av(2341,3\ol{5}241)$ (see \cite{west-thesis}).
\item Baxter permutations are $\av(41\ol{3}52,25\ol{3}14)$ (see \cite{gire}).
\item The elements of $\av(3142,2\ub{41}3)$ are in bijection with $\beta(1,0)$-trees (see \cite{claesson kitaev steingrimsson}).
\item By definition, alternating permutations are $\av(\ub{123},\ub{321},\ub{\pt12})$ and reverse alternating permutations are $\av(\ub{123},\ub{321},\ub{\pt21})$.
\end{itemize}
\end{ex}

There are some situations where vincular pattern avoidance can be equivalently phrased in terms of classical pattern avoidance, such as $\av(\ub{13}2) = \av(132)$.  However, the same cannot be said for barred pattern avoidance.

\begin{lem}\label{lem:barred not equiv to classical}
Let $p$ be a permutation.  There is no set $\ts{B}$ of barred patterns such that $\av(p) = \av(\ts{B})$.
\end{lem}

\begin{proof}
Suppose, to the contrary, that there is such a $\ts{B}$.  Let $\b \in \ts{B}$ have minimally many unbarred letters.  If $p$ is a permutation of $[n]$, then $\av(p)$ contains all permutations of $[k]$ for each $k<n$.  Thus $\b$ must have at least $n$ unbarred letters.  Because $p \not\in \av(p)$, the unbarred portion of $\b$ must actually be a $p$-pattern.  But then $\class{\b} \not\in \av(p) = \av(\ts{B}) \subseteq \av(\b)$, contradicting Lemma~\ref{lem:x in av(x)}.
\end{proof}

On the other hand, as pointed out in \cite{steingrimsson}, there are some situations where avoidance of a barred pattern is equivalent to avoidance of a suitably chosen vincular pattern.

\begin{ex}\label{ex:bondable}\
\begin{itemize}
\item $\av(41\ol{3}52) = \av(3\ub{14}2)$.
\item $\av(25\ol{3}14) = \av(2\ub{41}3)$.
\item $\av(21\ol{3}54) = \av(2\ub{14}3)$.
\end{itemize}
\end{ex}

The veracity of the three parts of Example~\ref{ex:bondable} is easy to show, and we prove one of them here.  Proofs of the other two are entirely analogous.

\begin{lem}\label{lem:bondable ex}
$\av(41\ol{3}52) = \av(3\ub{14}2)$.
\end{lem}

\begin{proof}
Suppose $w \in \av(41\ol{3}52)$.  Thus any occurrence of $3142$ in $w$ must have a relative value of $2.5$ sitting between the ``$1$'' and the ``$4$'' in that pattern.  Thus $w$ avoids the vincular pattern $3\ub{14}2$, and so $w \in \av(3\ub{14}2)$.  Hence $\av(41\ol{3}52) \subseteq \av(3\ub{14}2)$.

If $w \not\in \av(41\ol{3}52)$, then there is a $3142$-pattern in $w$ that is not part of any $41352$-pattern in $w$.  Choose this occurrence of $3142$ in $w$ so that the positions of ``$1$'' and ``$4$'' are as close together in $w$ as possible.  This ensures that no values less than the ``$2$,'' nor greater than the ``$3$,'' appear between this ``$1$'' and ``$4$.''  Because this $3142$-pattern is not part of a $41352$-pattern, no values between the ``$2$'' and the ``$3$'' appear in any of these positions either.  Thus the ``$1$'' and ``$4$'' must actually be adjacent, and so we have an occurrence of $3\ub{14}2$.  Hence $w \not\in \av(3\ub{14}2)$, and so $\av(41\ol{3}52) \supseteq \av(3\ub{14}2)$.
\end{proof}

In contrast to Example~\ref{ex:bondable}, it is not the case that avoidance of a barred pattern is always equivalent to avoidance of a finite set of vincular patterns.

\begin{lem}\label{lem:not bondable}
There is no finite set of vincular patterns $\ts{V}$ for which $\av(1\ol{2}3) = \av(\ts{V})$.
\end{lem}

\begin{proof}
Suppose, to the contrary, that there is such a finite set $\ts{V}$.

Because $1, 123 \in \av(1\ol{2}3)$ but $12 \not\in \av(1\ol{2}3)$, we must have $\ub{\pt 12\pt} \in \ts{V}$.  Moreover, because $123 \in \av(1\ol{2}3)$, this is the only element of $\ts{V}$ whose underlying classical pattern is $12$.

Fix an integer $n>0$ and suppose, inductively, that for all $0 \le k < n$, there is a unique $\mathsf{v_k} \in \ts{V}$ with $\class{\mathsf{v_k}} = k(k-1)\cdots 4312$, and that in fact $\mathsf{v_k} = \ub{\pt k(k-1)\cdots 4312\pt}$.

Consider $u = n(n-1)\cdots 4312$.  Because $u \not\in \av(1\ol{2}3)$, it must contain some $\v \in \ts{V}$.  Because $\av(1\ol{2}3) = \av(\ts{V})$, Lemma~\ref{lem:x in av(x)} implies $\v \not\in \av(1\ol{2}3)$.  Thus $\class{\v} = \class{\mathsf{v_k}}$, for some $k \le n$.  The inductive hypothesis implies that $k=n$, so $\class{\v} = u$ itself.  It remains only to determine the bonds in $\v$.

For any $i \in [n-2]$, inserting $(-1)0$ immediately before the $i$th letter in $u$ yields an element of $\av(1\ol{2}3)$.  This ensures that the $(i-1)$st and $i$th letters must be bonded in $\v$ (where the $9$th letter is the left endpoint).  Moreover,
$$n(n-1)\cdots 4312{(n+1)(n+2)} \in \av(1\ol{2}3),$$
so there must also be a bond between the rightmost letter in $\v$ and the right endpoint.  It now remains to determine whether $1$ is bonded to either of its neighbors in $\v$.  If at least one of these bonds does not exist, then $\v$ would be contained in 
$$(n+1)n\cdots 54123 \in \av(1\ol{2}3).$$
Thus $1$ must be bonded to both of its neighbors, and so $\v = \ub{\pt n(n-1)\cdots 4312\pt}$.  Therefore, by induction, $\ub{\pt n(n-1)\cdots 4312\pt} \in \ts{V}$ for all $n \ge 0$, contradicting the assumption that $\ts{V}$ was a finite set.
\end{proof}

In order to classify Lemmas~\ref{lem:bondable ex} and~\ref{lem:not bondable}, we introduce the following terminology.

\begin{defn}
Suppose that $\ts{B}$ is a finite set of barred patterns and that $\ts{V}$ is a finite set of vincular patterns.  If $\av(\ts{B}) = \av(\ts{V})$, then $\ts{B}$ and $\ts{V}$ are \emph{coincident}.  If $\ts{B}$ (respectively, $\ts{V}$) is a finite set of barred (respectively, vincular) patterns for which such a set $\ts{V}$ (respectively, $\ts{B}$) exists, then $\ts{B}$ (respectively, $\ts{V}$) is \emph{coincidental}.  If $\ts{B}$ (respectively, $\ts{V}$) is a coincidental set containing just a single element, then that pattern itself is \emph{coincidental}.
\end{defn}

In \cite{steingrimsson}, Steingr\'imsson posed the problem of classifying all coincidental barred patterns.

\begin{ex}
The barred patterns $41\ol{3}52$, $25\ol{3}14$, and $21\ol{3}54$ are each coincidental.  The vincular patterns $3\ub{14}2$, $2\ub{41}3$, and $2\ub{14}3$ are each coincidental.  The sets $\{41\ol{3}52\}$ and $\{3\ub{14}2\}$ are coincident.  The barred pattern $1\ol{2}3$ is not coincidental.
\end{ex}

In the present work, we will refine the notion of a coincidental barred permutation to what we call ``naturally coincidental,'' or ``nat-co.''  This concept will be made precise in Definition~\ref{defn:bondable}, and is intended to capture the particular prohibitions of the particular barred pattern.

Of the two characterizations, it is optimal to start with coincidental vincular patterns because, as shown in Proposition~\ref{prop:co vinc has |B|=1}, each such pattern is coincident with a single (and, as it turns out, naturally coincidental) barred pattern.  Thus both characterizations can be addressed by examining naturally coincidental barred patterns (Theorem~\ref{thm:main bv}).

\section{Preparation for coincidental vincular patterns}\label{section:prep vb}

\begin{defn}\label{defn:bar(v)}
For a vincular pattern $\v$, let $\bar{\v}$ be the set of barred patterns obtained by replacing a single bond of $\v$ by a barred symbol so that the result does not contain $\v$, and ignoring any remaining bonds.
\end{defn}

\begin{ex}\
\begin{itemize}
\item $\bar{3\ub{14}\ub{2\pt}} = \{41\ol{3}52,\ 4253\ol{1},\ 3152\ol{4},\ 3142\ol{4}\}$.
\item $\bar{3\ub{142}} = \{41\ol{3}52,\ 31\ol{4}52,\ 31\ol{5}42,\ 425\ol{1}3,\ 315\ol{4}2,\ 314\ol{5}2\}$.
\end{itemize}
\end{ex}

Lemma~\ref{lem:x in av(x)} and the definition of $\bar{\v}$ immediately imply the following result.

\begin{cor}\label{cor:bar is avoided}
For any vincular pattern $\v$,
$$\{\class{\b} : \b \in \bar{\v}\} \subseteq \big(\av(\v)\cap \av(\bar{\v})\big).$$
\end{cor}

Elements of $\bar{\v}$ are those barred patterns whose bars reflect the vincular nature of $\v$.  

\begin{prop}\label{prop:co vinc has |B|=1}
If $\v$ is a coincidental vincular pattern, then $|\bar{\v}| = 1$.
\end{prop}

\begin{proof}
Suppose, to the contrary, that $\av(\v) = \av(\ts{B})$ for some finite set $\ts{B}$ of barred patterns, and that there exist distinct $\b,\b' \in \bar{\v}$.  By Lemma~\ref{lem:x in av(x)}, we know that $\class{\b} \in \av(\b)$ and $\class{\b'} \in \av(\b')$.  Also, the definition of $\bar{\v}$ gives
\begin{equation}\label{eqn:b b'}
\class{\b},\class{\b'} \in \av(\v) = \av(\ts{B}).
\end{equation}

Suppose that $\v$ is a vincular pattern of $n$ letters.  Each barred pattern in $\ts{B}$ must contain at least $n$ unbarred letters, because $\av(\v) = \av(\ts{B})$ contains all permutations of fewer than $n$ letters.  The barred patterns $\b$ and $\b'$ each have $n$ unbarred letters and a single unbarred letter, meaning that they must be elements of $\ts{B}$ itself in order to satisfy equation~\eqref{eqn:b b'}.  However, because
$$\class{\b} \not\in \av(\b') \supseteq \big(\av(\b)\cap\av(\b')\big) \supseteq \av(\ts{B}) = \av(\v),$$
this contradicts equation~\eqref{eqn:b b'}.

If $\bar{\v} = 0$, then $\av(\v) = \av(\class{\v})$.  But, by Lemma~\ref{lem:barred not equiv to classical}, such a $\v$ would not be coincidental.

Thus $\bar{\v} = 1$.
\end{proof}

Suppose that $\v$ is a coincidental vincular pattern.  Then, by Proposition~\ref{prop:co vinc has |B|=1}, we have $\bar{\v} = \{\b\}$, and $\av(\v) = \av(\b)$.  This $\b$ has the particular form dictated by the definition of $\bar{\v}$, and it will be clear from Definition~\ref{defn:bondable} that this $\b$ is then a naturally coincidental vincular pattern.  Thus the problem of determining exactly which vincular patterns are coincidental is a special case of the problem of determining exactly which barred patterns are nat-co, which will be treated in Theorem~\ref{thm:main bv}.

\section{Preparation for (naturally) coincidental barred patterns}\label{section:prep bv}

We now define what it means for a barred pattern to be nat-co.

\begin{defn}\label{defn:bond(b)}
For a barred pattern $\b$, let $\bond{\b}$ be the set of vincular patterns obtained by using bonds to replace nonempty subsets of the barred letters in $\b$.
\end{defn}

Some of the vincular patterns obtained by the process described in Definition~\ref{defn:bond(b)} could be equivalent, as in the second example below.

\begin{ex}\label{ex:degeneracy}\
\begin{itemize}
\item $\bond{4\ol{13}5\ol{2}} = \{\ub{32}41,\ 3\ub{14}2,\ 312\ub{4\pt},\ \ub{23}1,\ 2\ub{13\pt},\ \ub{21}\ub{3\pt},\ \ub{12\pt}\}$.
\item $\bond{5\ol{1234}6} = \{ \ub{41}235,\ 4\ub{12}35,\ 41\ub{23}5,\ 412\ub{35},\ \ub{31}24,\ \ub{312}4,\ \ub{31}\ub{24},\ 3\ub{12}4,\ 3\ub{124},\ 31\ub{24},$ $\ub{21}3,\ \ub{213},\ 2\ub{13},\ \ub{12}\}$.
\end{itemize}
\end{ex}

The elements of $\bond{\b}$ are those vincular patterns whose bonds reflect the barred nature of $\b$.  As such, these will be the only vincular patterns we allow in our attempt to have vincular patterns mimic the behavior of a barred pattern.

The next result follows immediately from the definition of the set $\bond{\b}$, and is a complement to Corollary~\ref{cor:bar is avoided}.

\begin{cor}\label{cor:bond isn't avoided}
For any barred pattern $\b$,
$$\{\class{\v} : \v \in \bond{\b}\} \cap \av(\b) = \emptyset.$$
\end{cor}

\begin{defn}\label{defn:bondable}
A barred pattern $\b$ is \emph{naturally coincidental}, or \emph{nat-co}, if $\av(\b) = \av(\bond{\b})$.
\end{defn}

To show that $1\ol{2}3$ is not nat-co, it would suffice to prove that $\av(1\ol{2}3) \neq \av(\ub{12})$.  Thus Lemma~\ref{lem:not bondable} actually proves a stronger result about the barred pattern $1\ol{2}3$.

In order to characterize nat-co barred patterns, we will need to consider their maximal factors (that is, consecutive subsequences) of barred letters, together with the letters (if any) that appear immediately to the left and to the right of these factors.  To this end, we define ``boycotts,'' so-named to refer to the collective barring of a set of values.

\begin{defn}\label{defn:boycott}
Let $\b$ be a barred pattern.  Suppose that $\ol{\b_i\b_{i+1}\cdots \b_j}$ is a barred factor in $\b$, of maximal length.  Then $X = \{\b_{i-1},\b_i,\ldots,\b_j,\b_{j+1}\}$ is a \emph{boycott} of $\b$.  Let $U(X) = \{\b_{i-1},\b_{j+1}\}$ be the set of unbarred letters in $X$, and $B(X) = \{\b_i,\ldots,\b_j\}$ be the set of barred letters in $X$.
\end{defn}

Note that either element of the set $U(X)$ may be undefined, although they are not both undefined because we have assumed that all barred patterns are proper barred patterns.

\begin{ex}\label{ex:boycott}
The barred pattern $\ol{92}43\ol{71}5\ol{6}8$ has three boycotts.  From left to right, they are $\{2,4,9\}$, $\{1,3,5,7\}$, and $\{5,6,8\}$.  Moreover, $U(\{2,4,9\}) = \{4\}$ and $B(\{2,4,9\}) = \{2,9\}$.
\end{ex}

The definition of a boycott gives the following easy result.

\begin{lem}\label{lem:intersecting boycotts}
Any distinct boycotts in a barred pattern share at most one letter, and that letter is unbarred.
\end{lem}

\section{Main results}\label{section:main}

This paper gives a complete description of all naturally coincidental barred patterns, and all coincidental vincular patterns.  These results are stated as Theorem~\ref{thm:main bv}, and Corollary~\ref{cor:main vb}, respectively.

\begin{thm}\label{thm:main bv}
A barred pattern of $n$ letters is naturally coincidental if and only if it has a unique boycott $X$ and satisfies
\begin{itemize}
\item $|B(X)| \le 2$,
\item for all distinct $x,x' \in X$ with $\{x,x'\} \neq U(X)$, we have $|x-x'| > 1$, and
\item for all $0 \le k \le n$, we have $\{k,k+1\}\cap X \neq \emptyset$.
\end{itemize}
\end{thm}

Note that the last bullet point of Theorem~\ref{thm:main bv} means $\{1,n\} \in X$, and all three requirements force $3 \le n \le 7$.

The proof of this result will follow from a sequence of smaller results, and is presented in Section~\ref{section:proof}.

\begin{ex}\label{ex:coincidental}
It is easy to check that $1\ol{4}23$, $\ol{1}\ol{5}324$, and $63\ol{1}\ol{7}524$ each satisfy the hypotheses of Theorem~\ref{thm:main bv}.  The conclusion for each case is given below.
\begin{itemize}
\item $\av(1\ol{4}23) = \av(\ub{12}3)$.
\item $\av(\ol{1}\ol{5}324) = \av(\ub{\pt2}13,\ \ub{13}24,\ \ub{\pt4}213)$.
\item $\av(63\ol{1}\ol{7}524) = \av(5\ub{24}13,\ 63\ub{15}24,\ 5\ub{26}413)$.
\end{itemize}
\end{ex}

It is interesting to note that the non-boycott portion of a nat-co barred pattern is unspecified, once the hypotheses of Theorem~\ref{thm:main bv} are met.  For example, $\ol{15}324$ and $\ol{15}342$ are both nat-co barred patterns.

The restrictive nature of Theorem~\ref{thm:main bv} yields several easy corollaries.

\begin{cor}\label{cor:nat-co vinc precise}
A barred pattern of $n$ letters is nat-co if and only if it as a unique boycott $X$ and satisfies
\begin{itemize}
\item $n = 7$ and $X = \{1,3,5,7\}$, with $|B(X)| = 2$,
\item $n = 6$ and $(U(X), B(X))$ is one of the following pairs:
$$\big(\{1,2\}, \{4,6\}\big),\ \big(\{3,4\},\{1,6\}\big), \text{\ or } \big(\{5,6\},\{1,3\}\big),$$
\item $n = 5$ and $X = \{1,3,5\}$, with $|B(X)| \in \{1,2\}$,
\item $n = 4$ and $(U(X), B(X))$ is one of the following pairs:
$$\big(\{1,2\},\{4\}\big) \text{\ or } \big(\{3,4\}, \{1\}\big),$$
or
\item $n = 3$ and $X = \{1,3\}$, with $|B(X)|$ necessarily equal to $1$.
\end{itemize}
\end{cor}

Note that Corollary~\ref{cor:nat-co vinc precise} points out that for a nat-co barred pattern $\b$, there will be no degeneracy in the set $\bond{\b}$ as there had been in the second part of Example~\ref{ex:degeneracy}.

\begin{cor}\label{cor:size of vinc}
For any nat-co barred pattern $\b$, the size of the set $\bond{\b}$ is either $1$ or $3$.  More precisely, if $\b$ has a single barred symbol then $|\bond{\b}| = 1$, and if $\b$ has two barred symbols then $|\bond{\b}| = 3$.
\end{cor}

\begin{cor}\label{cor:total}
There are $720$ nat-co barred patterns.
\end{cor}

Corollary~\ref{cor:total} is refined in Table~\ref{table:bondable perms}.

\begin{table}[htbp]
\begin{tabular}{r||c|c|c|c|c}
& ${n=}$ 3 & 4 & 5 & 6 & 7\\
\hline
\hline
${b=}$ 1 & 4 & 8 & 36 & 0 & 0\\
\hline
2 & 0 & 0 & 24 & 72 & 576
\end{tabular}
\caption{The number of nat-co barred patterns having $n$ letters, of which $b$ are barred.  There are $0$ nat-co barred patterns for each $(n,b)$ not listed in the table.}\label{table:bondable perms}
\end{table}

An example of a nat-co barred permutation for each of the nonzero entries in Table~\ref{table:bondable perms} is given in Table~\ref{table:bondable perm ex}.

\begin{table}[htbp]
$$\begin{array}{r||c|c|c|c|c}
& {n=} 3 & 4 & 5 & 6 & 7\\
\hline
\hline & & & & & \\ [-2ex]
{b=} 1 & \ol{3}12 & 1\ol{4}23 & 1\ol{5}324 & - & -\\
\hline & & & & & \\ [-2ex]
2 & - & - & \ol{3}\ol{5}124 & 1\ol{4}\ol{6}235 & 1\ol{5}\ol{7}3246
\end{array}$$
\caption{Examples of nat-co barred patterns having $n$ letters, of which $b$ are barred.}\label{table:bondable perm ex}
\end{table}

The complementary result to Theorem~\ref{thm:main bv}, the characterization of coincidental vincular patterns, is now very easy to state.  It follows from Proposition~\ref{prop:co vinc has |B|=1}, and Definitions~\ref{defn:bar(v)} and~\ref{defn:bond(b)}, that a vincular pattern $\v$ is coincidental if and only if $\bar{\v} = \{\b\}$ for some nat-co barred pattern $\b$.

\begin{cor}\label{cor:main vb}
There are $48$ coincidental vincular patterns, all of which were yielded by Theorem~\ref{thm:main bv}.
\end{cor}

The $48$ patterns of Corollary~\ref{cor:main vb} are listed in Table~\ref{table:co vinc patterns}.  Note that the first entry in each of the three columns of Table~\ref{table:co vinc patterns} has its corresponding coincidental barred pattern in the $b=1$ row of Table~\ref{table:bondable perm ex}.

\begin{table}[htbp]
$$\begin{array}{c|cc|ccc}
n=2 & n=3 & & n=4\\
\hline
\hline
\ub{\pt1}2 & \ub{12}3 & 3\ub{12} & \ub{13}24 & 2\ub{13}4 & 24\ub{13}\\
\ub{\pt2}1 & \ub{21}3 & 3\ub{21} & \ub{13}42 & 4\ub{13}2 & 42\ub{13}\\
1\ub{2\pt} & \ub{23}1 & 1\ub{23} & \ub{31}24 & 2\ub{31}4 & 24\ub{31}\\
2\ub{1\pt} & \ub{32}1 & 1\ub{32} & \ub{31}42 & 4\ub{31}2 & 42\ub{31}\\
 & & & \ub{14}23 & 2\ub{14}3 & 23\ub{14}\\
 & & & \ub{14}32 & 3\ub{14}2 & 32\ub{14}\\
 & & & \ub{41}23 & 2\ub{41}3 & 23\ub{41}\\
 & & & \ub{41}32 & 3\ub{41}2 & 32\ub{41}\\
 & & & \ub{24}13 & 1\ub{24}3 & 13\ub{24}\\
 & & & \ub{24}31 & 3\ub{24}1 & 31\ub{24}\\
 & & & \ub{42}13 & 1\ub{42}3 & 13\ub{42}\\
 & & & \ub{42}31 & 3\ub{42}1 & 31\ub{42}
\end{array}$$
\caption{All coincidental vincular patterns, organized by number of letters, $n$, in the pattern.}\label{table:co vinc patterns}
\end{table}

\section{Proof of Theorem~\ref{thm:main bv}}\label{section:proof}

The proof of Theorem~\ref{thm:main bv} will be given in several steps.  We first show that each boycott in a nat-co barred pattern must contain both the smallest and largest letters in the pattern, and hence there is exactly one boycott.  We will then describe the minimal and maximal distance between values in the boycott, finally concluding that there are at most two barred letters in the boycott.

\begin{prop}\label{ref:boycott contents}
Let $\b$ be a nat-co barred pattern of $n$ letters, and $X$ a boycott in $\b$.  Then $1,n \in X$.
\end{prop}

\begin{proof}
Suppose, to the contrary, that $1 \not\in X$.  (The proof that $n \in X$ is entirely analogous.)

Write $U(X) = \{\b_{i-1},\b_{j+1}\}$ and $B(X) = \{\b_i,\b_{i+1},\cdots, \b_j\}$, where one or both of the elements of $U(X)$ may be undefined.  Let $\v \in \bond{\b}$ be the unique vincular pattern in $\bond{\b}$ having fewest letters, obtained from $\b$ by replacing all of its barred symbols by bonds.  Let $z$ be the classical permutation obtained by inserting a new smallest symbol between the bonded letters $\b_{i-1}$ and $\b_{j+1}$ in $\v$.  Note that
$$z \not\in \av(\b),$$
so $z$ must contain some element of $\bond{\b}$.

Because $1 \not\in B(X)$, we have that $z \not\in \{\class{\ts{v'}} : \ts{v'} \in \bond{\b}\}$.  Similarly, because $1 \not\in U(X)$, the permutation $z$ does not contain the vincular pattern $\v$.  Because $\v$ is the unique element of $\bond{\b}$ using the fewest letters, and because $z$ has just one more letter than $\v$ does, we see that $z$ does not contain any of the vincular elements of $\bond{\b}$ as patterns, and so
$$z \in \av(\bond{\b}).$$

This contradicts the assumption that $\av(\b) = \av(\bond{\b})$.
\end{proof}

The follows corollary is an immediate result of Lemma~\ref{lem:intersecting boycotts} and Propostion~\ref{ref:boycott contents}.

\begin{cor}
A nat-co barred pattern $\b$ has exactly one boycott.
\end{cor}

For the remainder of this section, suppose that $\b$ is a nat-co barred pattern of $n$ letters, with unique boycott $X$.  Moreover, let  $U(X) = \{\b_{i-1},\b_{j+1}\}$ and $B(X) = \{\b_i,\b_{i+1},\cdots, \b_j\}$.

Consider the proximity of values that appear in $X$.

\begin{prop}\label{prop:distance in B(X)}
For any distinct $x,x' \in B(X)$, we have $|x-x'| > 1$.
\end{prop}

\begin{proof}
Suppose, to the contrary, that there exist $i \le h < h' \le j$ with $|\b_h - \b_{h'}| = 1$.  Then
$$\b_1 \cdots \b_{h-1}\ub{\b_h \b_{h'+1}}\b_{h'+2} \cdots \b_n \in \bond{\b}$$
is contained in $\b$.  Thus $\b \not\in \av(\bond{\b}) = \av(\b)$, contradicting Lemma~\ref{lem:x in av(x)}.
\end{proof}

\begin{prop}\label{prop:distance between U and B}
For any $x \in B(X)$ and $x' \in U(X)$, we have $|x-x'| > 1$.
\end{prop}

\begin{proof}
Suppose, to the contrary, that there exists $\b_h \in B(X)$ such that (without loss of generality) $|\b_h - \b_{i-1}| = 1$.  Then
$$\b_1 \cdots \b_{i-2} \ub{\b_{i-1}\b_{h+1}}\b_{h+2}\cdots \b_n \in \bond{\b}$$
is contained in $\b$.  Thus $\b \not\in \av(\bond{\b}) = \av(\b)$, again contradicting Lemma~\ref{lem:x in av(x)}.
\end{proof}

The previous two propositions demonstrate that there is a minimal distance between values in $B(X)$, and between $B(X)$ and $U(X)$ values.  The next proposition shows that there is also a maximal distance between values in $X$.

\begin{prop}\label{prop:max distance between U and B}
For all $0 \le k \le n$, we have $\{k,k+1\} \cap X \neq \emptyset$.
\end{prop}

\begin{proof}
Suppose, to the contrary, that there is such a $k$ with $k,k+1\not\in X$.  Consider
$$z = \b_1 \cdots \b_{i-1} (k+.5) \b_{j+1}\cdots \b_n,$$
obtained by removing the entire barred factor $\b$ and inserting $k+.5$ in its place.  By construction, $z \not\in \{\class{\v} : \v \in \bond{\b}\}$, and $z$ does not contain the vincular pattern $\b_1 \cdots \ub{\b_{i-1} \b_{j+1}}\cdots \b_n \in \bond{\b}$.  Thus $z \in \av(\bond{\b})$.  However, $z \not\in \av(\b)$, yielding a contradiction.
\end{proof}

\begin{prop}\label{prop:boycott size}
$|B(X)| \le 2$.
\end{prop}

\begin{proof}
Suppose, to the contrary, that $|B(X)| > 2$; that is, $j-i > 1$.  Then consider
$$y = \b_1 \cdots \b_{i-1} \b_{i+1} (\b_{i+1} + .5) \b_{j+1} \cdots \b_n.$$
By Propositions~\ref{prop:distance in B(X)}--\ref{prop:max distance between U and B}, we have $y \in \av(\bond{\b})$.  However, once again we have $y \not\in \av(\b)$, yielding a contradiction.
\end{proof}

Propositions~\ref{prop:distance in B(X)}--\ref{prop:boycott size} now prove one direction of Theorem~\ref{thm:main bv}: if a barred pattern is nat-co then it must satisfy the listed points.  The other direction of the theorem is given by the following proposition.

\begin{prop}\label{prop:other direction}
If a barred pattern $\b$ has one of the forms outlined in Corollary~\ref{cor:nat-co vinc precise}, then $\b$ is nat-co.
\end{prop}

Because the rules are so restrictive, Proposition~\ref{prop:other direction} is easy to check by hand, and relies entirely on the following observation: inserting a value to break the bond in some $\v \in \bond{\b}$ will always yield a permutation that again contains a pattern from $\bond{\b}$, or else is equal to $\b$ itself.  The proof of Lemma~\ref{lem:bondable ex} demonstrates the procedure.

\section{Further directions}

There are two obvious directions for future research.  The first is to characterize all sets of coincidental barred or vincular patterns --- not just those of cardinality $1$.  Certainly there are such sets, as indicated by Example~\ref{ex:coincidental}.  Also, although our imposed ``natural'' coincidence restriction for barred patterns was entirely in keeping with their barred behavior, it could be interesting to try to characterized all coincidental barred patterns after relaxing this rule.  Issues of pattern containment can be especially fussy, and so it is not clear that completely removing the naturality condition would be productive, but there may be some intermediate requirement that would be interesting.

\end{document}